\documentclass[11pt]{amsart}

\usepackage{amsfonts, amstext, amsmath, amsthm, amscd, amssymb}
\usepackage{graphicx, color, epsfig, subfigure, wrapfig, overpic}
\usepackage[all,cmtip]{xy} 

\let\oldmarginpar\marginpar
\renewcommand\marginpar[1]{\oldmarginpar[\raggedleft\footnotesize #1]%
{\raggedright\footnotesize #1}}

\AtBeginDocument{%
   \def\MR#1{}
}

\newcommand{\W}{\mathcal W}
\newcommand{\vol}{{\rm vol}}

\newcommand{\specvol}{{\rm Spec}_{\vol}}
\newcommand{\specdet}{{\rm Spec}_{\rm det}}
\newcommand{\specjones}{{\rm Spec}_{\rm JP}}
\newcommand{\specQ}{{\rm Spec}_q}
\newcommand{\specKH}{{\rm Spec}_{\rm KH}}
\newcommand{\voct}{{v_{\rm oct}}}

\theoremstyle{plain}
\newtheorem{theorem}{Theorem}[section]
\newtheorem{corollary}[theorem]{Corollary}

\newtheorem{prop}[theorem]{Proposition}

\newtheorem{conjecture}[theorem]{Conjecture}

\newtheorem*{namedtheorem}{\theoremname}
\newcommand{\theoremname}{testing}

\theoremstyle{definition}
\newtheorem{define}[theorem]{Definition}
\newtheorem{question}[theorem]{Question}
\newtheorem{remark}[theorem]{Remark}

\title[Density spectra for knots]{Density spectra for knots}
\author[A.\ Champanerkar]{Abhijit Champanerkar}
\address{Department of Mathematics, College of Staten Island \& The Graduate Center, City University of New York, New York, NY}
\email{abhijit@math.csi.cuny.edu}
\author[I. \ Kofman]{Ilya Kofman}
\address{Department of Mathematics, College of Staten Island \& The Graduate Center, City University of New York, New York, NY}
\email{ikofman@math.csi.cuny.edu}
\author[J. \ Purcell]{Jessica S.\ Purcell}
\address{Department of Mathematics, Brigham Young University, Provo, UT}
\email{jpurcell@math.byu.edu}

\begin{document}

\begin{abstract}
We recently discovered a relationship between the volume density
spectrum and the determinant density spectrum for infinite sequences
of hyperbolic knots.  Here, we extend this study to new quantum
density spectra associated to quantum invariants, such as Jones
polynomials, Kashaev invariants and knot homology.  We also propose
related conjectures motivated by geometrically and diagrammatically
maximal sequences of knots.
\end{abstract}

\maketitle

\centerline{\em In celebration of J\'ozef Przytycki's 60th birthday}

\section{Volume and determinant density spectra}
In \cite{ckp:gmax}, we studied the asymptotic behavior of two basic
quantities, one geometric and one diagrammatic, associated to an alternating
hyperbolic link $K$: The {\em volume density} of $K$ is defined as
$\vol(K)/c(K)$, and the {\em determinant density} of $K$ is defined as
$2\pi\log\det(K)/c(K)$.

For any diagram of a hyperbolic link $K$, an upper bound for the
hyperbolic volume $\vol(K)$ was given by D.~Thurston by decomposing
$S^3-K$ into octahedra at crossings of $K$.  Any hyperbolic octahedron
has volume bounded above by the volume of the regular ideal
octahedron, $\voct \approx 3.66386$.  So if $c(K)$ is the crossing
number of $K$, then
\begin{equation}\label{eq:dylan}
\frac{\vol(K)}{c(K)} \leq \voct. 
\end{equation} 

The following conjectured upper bound for the determinant density is
equivalent to a conjecture of Kenyon \cite{kenyon} for planar graphs.
We have verified this conjecture for all knots up to 16 crossings.
\begin{conjecture}[\cite{ckp:gmax}]\label{conj:diagmax}
If $K$ is any knot or link, $\displaystyle \ \frac{2\pi\log\det(K)}{c(K)} \leq \voct.$
\end{conjecture}

This motivates a more general study of the spectra for volume and determinant density.

\begin{define}\label{def:spec}
  Let $\mathcal{C}_{\rm vol} = \{\vol(K)/c(K)\}$ and $\mathcal{C}_{\rm
    det}= \{2\pi\log\det(K)/c(K)\}$ be the sets of respective
  densities for all hyperbolic links $K$.  We define $\specvol=
  \mathcal{C}'_{\rm vol}$ and $\specdet=\mathcal{C}'_{\rm det}$ as
  their derived sets (set of all limit points).
\end{define}

Equation (\ref{eq:dylan}) and Conjecture \ref{conj:diagmax} imply
$$\specvol, \ \specdet \subset [0,\voct]$$
Twisting on two strands of an alternating link gives $0$ as a limit
point of both densities: 
$\displaystyle 0\in\specvol \cap \specdet$.
Moreover, by the upper volume bound established in \cite{AdamsBound}, $\voct$ cannot occur as a
volume density of any finite link; i.e., $\voct\notin\mathcal{C}_{\rm vol}$.
However, $\voct$ is the volume density of the \emph{infinite weave} $\W$, the infinite alternating link with the
infinite square grid projection graph (see \cite{ckp:gmax}).  

To study $\specvol$ and $\specdet$, we consider sequences of knots and links.
We say that a sequence of links $K_n$ with $c(K_n)\to \infty$ is \emph{geometrically maximal} if
$\displaystyle \lim_{n\to\infty}\frac{\vol(K_n)}{c(K_n)}=\voct. $
Similarly, it is {\em diagrammatically maximal} if
$\displaystyle \lim_{n\to\infty}\frac{2\pi\log\det(K_n)}{c(K_n)}=\voct. $
In \cite{ckp:gmax}, we found many families
of geometrically and diagrammatically maximal knots and links that are
related to the infinite weave $\W$.

\begin{define}\label{def:folner}
Let $G$ be any possibly infinite graph. For any finite subgraph $H$,
the set $\partial H$ is the set of vertices of $H$ that share an edge
with a vertex not in $H$.  We let $|\cdot|$ denote the number of
vertices in a graph.  An exhaustive nested sequence of
connected subgraphs, $\{H_n\subset G:\; H_n\subset H_{n+1},\; \cup_n
H_n=G\}$, is a {\em F{\o}lner sequence} for $G$ if
$$ \lim_{n\to\infty}\frac{|\partial H_n|}{|H_n|}=0. $$
\end{define}

For any link diagram $K$, let $G(K)$ be the projection graph of
the diagram. Let $G(\W)$ be the projection graph of $\W$, which is the
infinite square grid.  We will need a particular diagrammatic
condition called a {\em cycle of tangles}, which is defined in
\cite{ckp:gmax}.  For an example, see Figure \ref{fig:celtic}.

\begin{figure}
\begin{tabular}{ccc}
  \includegraphics[height=1.2in]{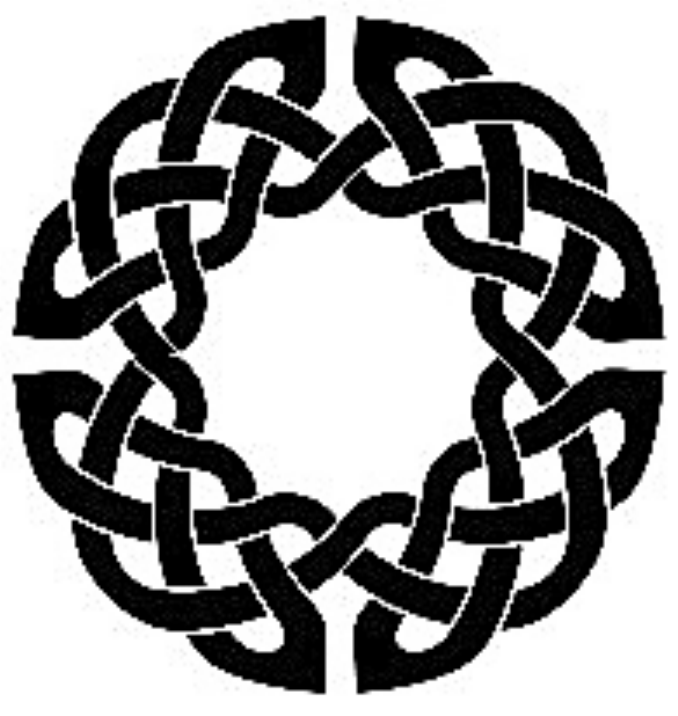} & \hspace*{0.5in} &
 \includegraphics[height=1.2in]{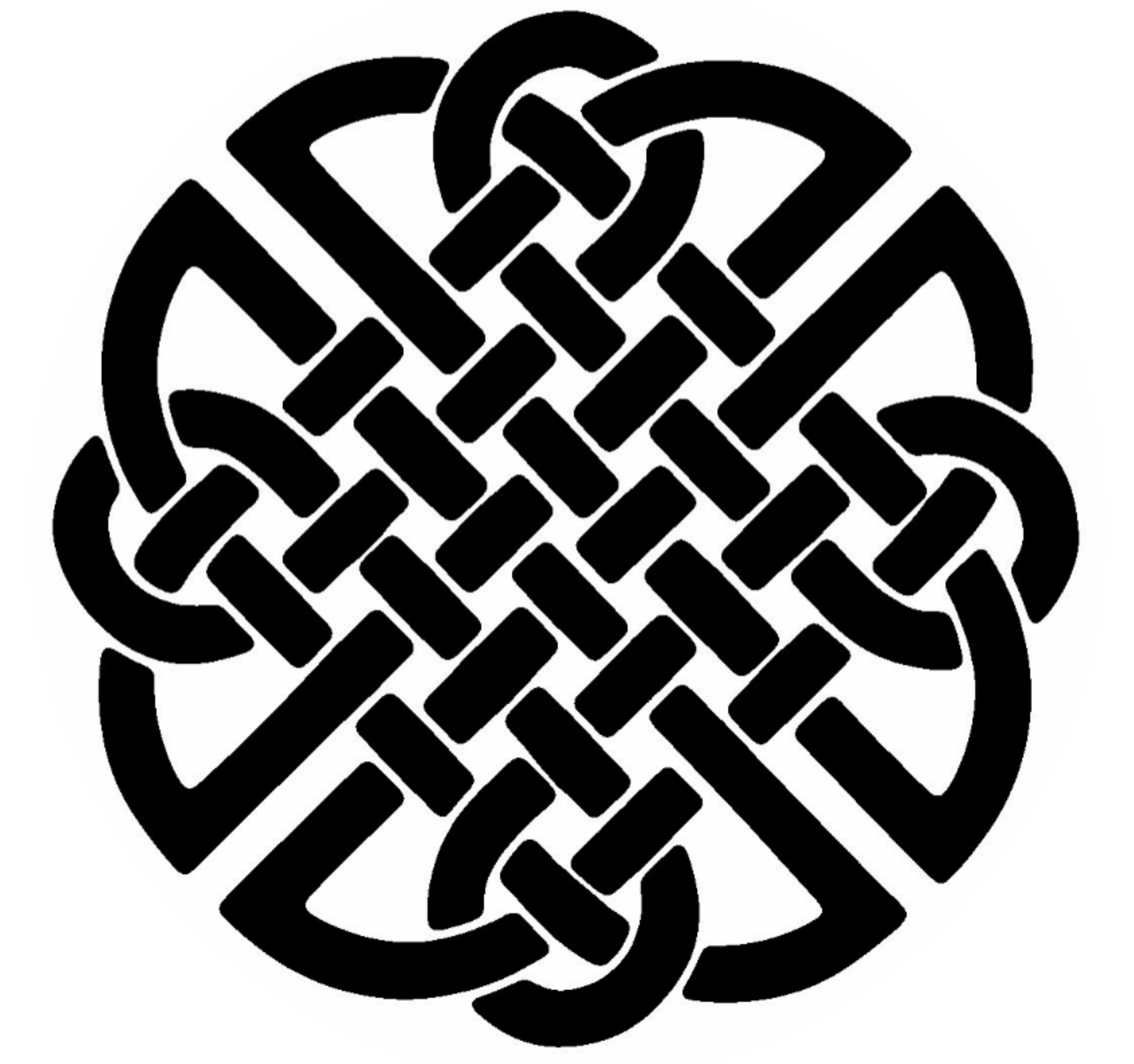} \\
(a) & \qquad & (b) \\
\end{tabular}
\caption{(a) A Celtic knot diagram that has a cycle of tangles.
  (b) A Celtic knot diagram with no cycle of tangles, which could be in a sequence that
satisfies conditions of Theorem~\ref{thm:geommax}.}
  \label{fig:celtic}
\end{figure}

\begin{theorem}[\cite{ckp:gmax}]\label{thm:geommax}
Let $K_n$ be any sequence of hyperbolic alternating links that contain no cycle of tangles, such that
\begin{enumerate}
\item there are subgraphs $G_n\subset G(K_n)$ that form a F\o lner sequence for $G(\W)$, and
\item $\lim\limits_{n\to\infty} |G_n|/ c(K_n) = 1$.
\end{enumerate}
Then $K_n$ is geometrically maximal:
$\displaystyle \lim_{n\to\infty}\frac{\vol(K_n)}{c(K_n)}=\voct. $
\end{theorem}

\begin{theorem}[\cite{ckp:gmax}]\label{thm:diagmax}
Let $K_n$ be any sequence of alternating links such that
\begin{enumerate}
\item there are subgraphs $G_n\subset G(K_n)$ that form a F\o lner sequence for $G(\W)$, and
\item $\lim\limits_{n\to\infty} |G_n|/ c(K_n) = 1$.
\end{enumerate}
Then $K_n$ is diagrammatically maximal:
$\displaystyle \lim_{n\to\infty}\frac{2\pi\log\det(K_n)}{c(K_n)} = \voct. $
\end{theorem}

Many families of knots and links are both geometrically and
diagrammatically maximal.  For example, weaving knots are alternating
knots with the same projection as torus knots, and are both
geometrically and diagrammatically maximal \cite{ckp:weaving, ckp:gmax}.
These results attest to the non-triviality of $\specvol \cap
\specdet$:
\begin{corollary}\label{cor:v8}
$\{0, \voct\} \subset \specvol \cap \specdet$.
\end{corollary}
It is an interesting problem to understand the sets $\specvol$,
$\specdet$, and $\specvol \cap \specdet$, and to explicitly describe
and relate their elements.  Below, we prove a new result that
$\specvol$ and $\specdet$ contain many interesting elements.  In
Section \ref{sec:quantum}, we extend these ideas to spectra related to
quantum invariants of knots and links.


\subsection{Volume and determinant density spectra from alternating links}

For any reduced alternating diagram $D$ of a hyperbolic alternating
link $K$, Adams \cite{adams:augmented} recently defined the following
notion of a {\em generalized augmented link} $J$.  Take an
unknotted component $B$ that intersects the projection sphere of $D$ in
exactly one point in each of two non-adjacent regions of $D$.  Then
$J=K\cup B$.  In \cite[Theorem~2.1]{adams:augmented}, Adams proved
that any such generalized augmented link is hyperbolic.

\begin{theorem}\label{thm:alt}
For any hyperbolic alternating link $K$, 
\begin{itemize}
\setlength\itemsep{0.5em}
\item[(a)] if $K \cup B$ is any generalized augmented alternating link, $\displaystyle \vol(K\cup B)/c(K)\in\specvol$,
\item[(b)] $\displaystyle 2\pi\log\det(K)/c(K)\in\specdet$.
\end{itemize}
\end{theorem}

\begin{proof}[Proof of part (a).]

View $K$ as a knot in the solid torus $S^3-B$. Cut along the disk bounded by $B$ (cutting $K$ each time $K$ intersects the disk bounded by $B$), obtaining a tangle $T$. Let $K^n$ denote the $n$--periodic reduced alternating link with quotient $K$, formed by taking $n$ copies of $T$ joined in an $n$--cycle of tangles as in Figure~\ref{fig:cycle-tangles}.  Thus,
$K^1=K$ and $c(K^n)=n\cdot c(K)$.

\begin{figure}
  \input{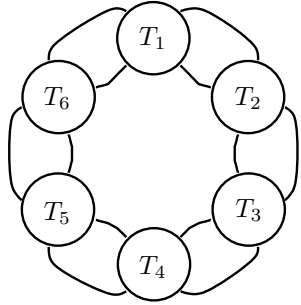}
  \caption{A 6--cycle of 2--tangles.}
  \label{fig:cycle-tangles}
\end{figure}

Let $B$ also denote the central axis of rotational symmetry of $K^n$. Then \cite[Theorem~3.1]{fkp:symmetric}, using results of \cite{fkp:volume}, implies that
\[ n\left( 1-\frac{2\sqrt{2}\,\pi^2}{n^2}\right)^{3/2} \vol(K\cup B) \: \leq \:
\vol(K^n) \: \leq \: n\,\vol(K\cup B).
\]

Therefore,
\[
\lim_{n\to\infty}\frac{\vol(K^n)}{c(K^n)} = \lim_{n\to\infty}\frac{n\cdot\vol(K\cup B)}{n\cdot c(K)} = \frac{\vol(K\cup B)}{c(K)}. 
\]
This completes the proof of part (a).
\end{proof}

For the proof of part (b), we recall some notation. Any alternating link $K$ is determined up to mirror image by its Tait graph $G(K)$, the planar checkerboard graph for which a vertex is assigned to every shaded region and an edge to every crossing of $K$. Thus, $e(G)=c(K)$.  Let $\tau(G)$ denote the number of spanning trees of $G$.  For any alternating knot, $\tau(G)=\det(K)$, which is the determinant of $K$ \cite{thistlethwaite}.  

We will need the following special case of \cite[Corollary~3.8]{Lyons}. Let $V(G)$ denote the set of vertices of $G$, and let $|G|$ denote the number of vertices.
\begin{prop}\label{prop:lyons}
Given $d>0$, let $G_n$ be any sequence of finite connected graphs with degree at most $d$ such that $\displaystyle \lim_{n\to\infty}\frac{\log\tau(G_n)}{|G_n|} = h$.
If $G_n'$ is a sequence of connected subgraphs of $G_n$ such that 
\[
\lim_{n\to\infty}\frac{\# \{x \in V(G_n'):\; \deg_{G_n'}(x)=\deg_{G_n}(x)\} }{|G_n|} = 1,
\]
then \ 
$\displaystyle \lim_{n\to\infty}\frac{\log\tau(G_n')}{|G_n'|} = h$.
\end{prop} 

\begin{proof}[Proof of Theorem~\ref{thm:alt} part (b).]
As in the proof of part (a), let $K^n$ denote the $n$--periodic link formed by an $n$--cycle of tangles $T$.
Let $L^n=K\#\cdots\# K$ denote the connect sum of $n$ copies of $K$, which has a reduced alternating diagram as the closure of $n$ copies of $T$ joined in a row.
Note that $c(K^n)=c(L^n)=n\cdot c(K)$, and $\det(L^n)=(\det(L))^n$.

In terms of Tait graphs, $G(K^n)$ is obtained from $G(L^n)$ by identifying one pair of vertices, so that $G(L^n)$ is a subgraph of $G(K^{n+1})$, and $|G(L^n)|=|G(K^n)|+1$.  Hence, by Proposition~\ref{prop:lyons},
\[
\lim_{n\to\infty}\frac{\log\tau(G(K^n))}{|G(K^n)|} = \lim_{n\to\infty}\frac{\log\tau(G(L^n))}{|G(L^n)|}.
\]
Therefore,
\[
\lim_{n\to\infty}\frac{2\pi\log\det(K^n)}{c(K^n)} =
\lim_{n\to\infty}\frac{2\pi\log\det(L^n)}{c(L^n)} =
\lim_{n\to\infty}\frac{n\cdot 2\pi\log\det(K)}{n\cdot c(K)} =
\frac{2\pi\log\det(K)}{c(K)}.
\]
This completes the proof of part (b).
\end{proof}

Note that part (a) of Theorem~\ref{thm:alt} generalizes \cite[Corollary~3.7]{ckp:weaving}, where $B$ was the braid axis.


\section{Quantum density spectra}\label{sec:quantum}

\subsection{Jones polynomial density spectrum}
Let $V_K(t)=\sum_i a_i t^i$ denote the Jones polynomial, with $d={\rm{span}}(V_K(t))$, which is the difference between the highest and lowest degrees of terms in $V_K(t)$.  Let $\mu(K)$ denote the average of the absolute values of coefficients of $V_K(t)$, i.e.\
\[ \mu(K)= \frac{1}{d+1}\sum |a_i|. \]

For sequences of alternating diagrammatically maximal knots, we have:

\begin{prop}\label{prop:jones}
If $K_n$ is any sequence of alternating diagrammatically maximal links,
\[ \lim_{n\to\infty}\frac{2\pi\log\mu(K_n)}{c(K_n)}=\voct. \]
\end{prop}
\begin{proof}
If, as above, $G$ is the Tait graph of $K$, and $\tau(G)$ is the number of spanning trees, then $\tau(G)=\det(K)$ and $e(G)=c(K)$.
It follows from the spanning tree expansion for $V_K(t)$ in
\cite{thistlethwaite} that if $K$ is an alternating link,
$$ \mu(K) = \frac{\det(K)}{c(K)+1}. $$
Thus, 
$\displaystyle \frac{\log\mu(K)}{c(K)} = \frac{\log\det(K) - \log(c(K)+1)}{c(K)}$, and the result follows since $K_n$ are diagrammatically maximal links.
\end{proof}

We conjecture that the alternating condition in Proposition \ref{prop:jones} can be dropped.

\begin{conjecture}\label{conj:jonesmax}
If $K$ is any knot or link,
\[ \frac{2\pi\log\mu(K)}{c(K)} \leq \voct.\]
\end{conjecture}

\begin{prop}\label{prop:equivalent}
Conjecture~\ref{conj:diagmax} implies Conjecture~\ref{conj:jonesmax}.
\end{prop}
\begin{proof}
By the proof of Proposition \ref{prop:jones},
Conjecture~\ref{conj:diagmax} would immediately imply that
Conjecture~\ref{conj:jonesmax} holds for all alternating links $K$. By
the spanning tree expansion for $V_K(t)$, $\Sigma |a_i| \leq
\tau(G(K))$, with equality if and only if $K$ is alternating. Hence, if
$K$ is not alternating, then there exists an alternating link with the
same crossing number and strictly greater coefficient sum $\Sigma
|a_i|$. Therefore, Conjecture~\ref{conj:diagmax} would still imply
Conjecture~\ref{conj:jonesmax} in the non-alternating case.
\end{proof}

\begin{define}\label{def:specjones}
  Let $\mathcal{C}_{\rm JP}= \{2\pi\log\mu(K)/c(K)\}$ be the set of
  Jones polynomial densities for all links $K$.  We define
  $\specjones= \mathcal{C}'_{\rm JP}$ as its derived set (set of
  all limit points).
\end{define}

Conjecture \ref{conj:jonesmax} is that $\displaystyle \specjones \subset [0,\voct]$.  By the results above, we have
\begin{corollary}\label{cor:qv8}
$\{0, \voct\} \subset \specvol \cap \specdet \cap \specjones$.
\end{corollary}
\begin{proof}
Twisting on two strands of an alternating link gives $0$ as a common limit
point.  For links $K_n$ that satisfy Theorem~\ref{thm:geommax}, their asymptotic volume density equals their asymptotic determinant density, so in this case,
\[ \lim_{n\to\infty}\frac{\vol(K_n)}{c(K_n)} \: = \: \lim_{n\to\infty}\frac{2\pi\log\det(K_n)}{c(K_n)} \: = \: \lim_{n\to\infty}\frac{2\pi\log\mu(K_n)}{c(K_n)} \: = \: \voct. \]
\end{proof}


\subsection{Knot homology density spectrum}
For alternating knots, the ranks of their reduced Khovanov homology
and their knot Floer homology both equal the determinant.
Every non-alternating link can be viewed as a modification of a
diagram of an alternating link with the same projection, by changing
crossings. In \cite{ckp:gmax}, we proved that this modification affects the determinant as follows.
Let $K$ be a reduced alternating link diagram, and $K'$ be obtained by changing any proper subset of crossings of $K$.  Then
\[ \det(K') < \det(K). \]
Motivated by this fact, the first two authors previously conjectured
that alternating diagrams similarly maximize hyperbolic volume in a
given projection.  With $K$ and $K'$ as before, they have verified
that for all alternating knots up to 18 crossings ($\approx 10.7$
million knots),
\[ \vol(K') < \vol(K). \]
We conjecture in \cite{ckp:gmax} that the same result holds if $K'$ is
obtained by changing any proper subset of crossings of $K$.  Note that
by Thurston's Dehn surgery theorem, the volume converges from below
when twisting two strands of a knot, so $\vol(K) - \vol(K')$ can be an
arbitrarily small positive number.

We have verified the following conjecture for all alternating knots up to 16 crossings, and weaving knots and links for $3\leq p\leq 50$ and $2\leq q\leq 50$.
\begin{conjecture}[\cite{ckp:gmax}]\label{conj:detvol}
For any alternating hyperbolic link $K$, 
\[ \vol(K) < 2\pi\log\det(K). \]
\end{conjecture}

Conjectures \ref{conj:diagmax} and \ref{conj:detvol} would imply that any geometrically maximal sequence of knots is diagrammatically maximal.
In contrast, we can obtain $K_n$ by twisting on two strands, such that $\vol(K_n)$ is bounded but $\det(K_n)\to\infty$.
We also showed in \cite{ckp:gmax} that the inequality in Conjecture \ref{conj:detvol} is sharp, in the sense that
if $\alpha<2\pi$, then there exist alternating hyperbolic knots $K$ such that $\alpha\log\det(K)<\vol(K)$.

A natural extension of Conjecture~\ref{conj:detvol} to any hyperbolic
knot is to replace the determinant with the rank of the reduced
Khovanov homology $\widetilde{H}^{*,*}(K)$.  
We have verified the following conjecture for all
non-alternating knots with up to 15 crossings. 
\begin{conjecture}[\cite{ckp:gmax}]\label{conj:KHvol}
For any hyperbolic knot $K$, 
\[ \vol(K) < 2\pi\log\text{\rm rank}(\widetilde{H}^{*,*}(K)). \]
\end{conjecture}
\noindent
Note that Conjecture~\ref{conj:detvol} is a special case of Conjecture~\ref{conj:KHvol}.

\begin{question} 
Is Conjecture~\ref{conj:KHvol} true for knot Floer homology; i.e., is it true that\\ $\displaystyle \vol(K) < 2\pi\log\text{\rm rank}(HFK(K))\,?$
\end{question}

\begin{define}\label{def:specKH}
Let $\displaystyle \mathcal{C}_{\rm KH}=\{2\pi\log\text{\rm rank}(\widetilde{H}^{*,*}(K))/c(K)\}$ be the set of Khovanov homology densities for
all links $K$.  We define $\specKH= \mathcal{C}'_{\rm KH}$ as its derived set
(set of all limit points).
\end{define}

\begin{prop}\label{prop:KH} If $\specdet \subset [0,\voct]$ then $\specKH \subset [0,\voct]$.
\end{prop}
\begin{proof}
For alternating knots, $\text{rank}(\widetilde{H}^{*,*}(K))=\det(K)$.
Let $K$ be an alternating hyperbolic knot, and $K'$ be obtained by changing any proper subset of crossing of $K$.
It follows from results in \cite{KH} that 
$\displaystyle \det(K') \leq \text{rank}(\widetilde{H}^{*,*}(K'))\leq \det(K)$.
\end{proof}

\begin{question} 
Does $\specKH=\specdet$?
\end{question}


\subsection{Kashaev invariant density spectrum}

The Volume Conjecture (see, e.g.\ \cite{InteractionsBook} and
references therein) is an important mathematical program to bridge the
gap between quantum and geometric topology.  One interesting
consequence of our discussion above is a {\em maximal volume
  conjecture} for a sequence of links that is geometrically and
diagrammatically maximal.

The Volume Conjecture involves the Kashaev invariant
\[ \langle K \rangle_N := \frac{J_N(K; \exp(2\pi i/N))}{J_N(\bigcirc; \exp(2\pi i/N))},\]
and is the following limit:
\[ \lim_{N\to\infty} 2\pi\log|\langle K \rangle_N|^{\frac{1}{N}} = \vol(K). \]

For any knot $K$, Garoufalidis and Le \cite{GarLe} proved 
\[ \limsup_{N\to\infty}\frac{2\pi\log|\langle K\rangle_N|^{\frac{1}{N}}}{c(K)} \leq \voct \]

Now, since the limits in Theorems \ref{thm:geommax} and \ref{thm:diagmax} are both equal to $\voct$, we can make the maximal volume conjecture as follows.
\begin{conjecture}[Maximal volume conjecture]\label{asympvc}
For any sequence of links $K_n$ that is both geometrically and diagrammatically maximal, there exists an increasing integer-valued function $N=N(n)$ such that
\[ \lim_{n\to\infty}\frac{2\pi\log|\langle K_n\rangle_N|^{\frac{1}{N}}}{c(K_n)} \: = \: \voct \: = \: \lim_{n\to\infty} \frac{\vol(K_n)}{c(K_n)}. \]
\end{conjecture}

To prove Conjecture \ref{asympvc} it suffices to prove 
\[ \lim_{n\to\infty}\frac{2\pi\log|\langle K_n\rangle_N|^{\frac{1}{N}}}{c(K_n)} \: = \:  \lim_{n\to\infty}\frac{2\pi\log\det(K_n)}{c(K_n)} \: = \: \voct, \]
which relates only diagrammatic invariants.

These ideas naturally suggest an interesting quantum
density spectrum:
\begin{define}\label{def:specQ}
Let $\displaystyle \mathcal{C}_q=\{2\pi\log|\langle
K\rangle_N|^{\frac{1}{N}}/c(K),\, N\geq 2\}$ be the set of quantum
densities for all links $K$ and all $N\geq 2$.  We define $\specQ=
\mathcal{C}'_q$ as its derived set (set of all limit points).
\end{define}

Conjecture \ref{asympvc} would imply that $\voct \in \specQ$.  The
Volume Conjecture would imply that $\specvol \subset \specQ$.

\begin{remark}
For every link $K$ for which the Volume Conjecture holds,
$\displaystyle \frac{\vol(K)}{c(K)}\in\specQ$.  In particular, since the Volume
Conjecture has been proved for torus knots, the figure-eight knot, Whitehead link and Borromean link (see \cite{hitoshi}), we know that
certain rational multiples of volumes of the regular ideal tetrahedron
and octahedron are in $\specQ$; namely, $$\{0,\, \frac{1}{2}v_{\rm tet},\, \frac{1}{5}\voct,\, \frac{1}{3}\voct\} \subset \specQ.$$
\end{remark}

If $N=2$, then $|\langle K\rangle_N|=\det(K)$, so $\frac{1}{2}\specdet \subset \specQ$.  

\ 

Taken together, the results above suggest the following general conjecture:

\begin{conjecture}
$$ \specvol = \specdet = \specQ = [0,\voct]. $$
\end{conjecture}

\ 

\subsection*{Acknowledgements}
The first two authors acknowledge support by the Simons Foundation and PSC-CUNY. The third author acknowledges support by the National Science Foundation under grant number DMS--1252687.  

\bibliographystyle{amsplain} \bibliography{references}

\end{document}